\documentclass{article}
\usepackage{amsmath}
\usepackage{amsfonts}
\usepackage{amsthm}
\usepackage{enumerate}
\usepackage{amssymb}

\newtheorem{thm}{Theorem}[section]
\newtheorem{prop}[thm]{Proposition}
\newtheorem{cor}[thm]{Corollary}
\newtheorem{lem}[thm]{Lemma}

\newtheorem*{thmnn}{Theorem}

\theoremstyle{remark}
\newtheorem{rmk}[thm]{Remark}

\theoremstyle{definition}

\begin{document}

\title{Seesaw Identities and Theta Contractions with Generalized Theta Functions, and Restrictions of Theta Lifts}

\author{Shaul Zemel}

\maketitle

\section*{Introduction}

Theta lifts from elliptic modular forms to various types of functions on Grassmannians (or symmetric spaces of orthogonal groups) were defined systematically in \cite{[B]}. They have numerous applications, including in Number Theory and Algebraic Geometry (e.g., the Kudla program). These come from the Hermitian structure on these Grassmannians when the signature is $(2,n)$ (or $(n,2)$), but the lifts are defined for arbitrary signatures.

A primitive non-degenerate sub-lattice $M$ of a fixed lattice $L$ yields, when complemented with an element of the Grassmannian of $M^{\perp}$, an embedding from the Grassmannian of $M$ into that $L$. One may therefore consider the restriction of functions on the Grassmannian of $L$ to that of $M$ (perhaps modified to avoid global singularities). When both Grassmannians are Hermitian, these restrictions were considered in \cite{[M]} for Borcherds products. This reference shows that after removing the global singularities, the restriction of the Borcherds product associated with a modular form $F$ is the one associated with the \emph{theta contraction} of $F$, which is formed from $F$ using the theta function associated with the theta function associated with the (positive definite) orthogonal complement $M^{\perp}$. However, the proof there is based on analyzing divisors of Borcherds products, and therefore require the assumption that the Koecher principle holds.

\smallskip

The main goal of this paper is to consider theta lifts and their restrictions, which is done in the generality of the basic constructions from \cite{[B]}. Along the way we generalize the seesaw identity to theta functions with polynomials, and also find a way to express the theta contractions from \cite{[M]} (generalized to this setting as well) as a pairing of $F$ with a slightly modified theta function. The latter function appears, in some particular settings, already in \cite{[B]} and \cite{[BZ]}, where it is claimed to be a modular form. We prove this property in general here. In fact, we establish the transformation formula involving the two additional parameters $\alpha$ and $\beta$ from \cite{[B]}, which can be related to Jacobi forms with lattice indices that are not necessarily positive definite. We study the latter objects at a subsequent publication.

The main result about restrictions of theta lifts is the following one.
\begin{thmnn}
Let $M$ be a primitive non-degenerate sub-lattice $M$ of a lattice $L$, take an element of the Grassmannian of $M^{\perp}$, and use it to embed the Grassmannian of $M$ into that of $L$. Denote by $\Phi_{L}$ the function on the Grassmannian of $L$ arising as the theta lift of a modular form $F$ and representation $\rho_{L}^{*}$ using a theta function with a polynomial. Then if we subtract the singular terms of $\Phi_{L}$ that cover the Grassmannian of $M^{\perp}$ completely, then the restriction of the resulting function to the Grassmannian of $M$ is the theta lift $\Phi_{M}$ of the theta contraction of $F$ with the appropriate theta function associated with $M^{\perp}$.
\end{thmnn}
This result holds for all the possible signatures, both of $L$ and of $M$, and depends on the modularity of the modified theta functions used in the definition of the theta contraction. It is obtained by comparing the pairing appearing in the definition of the theta lift using the expression coming from the theta contraction. In the case where both Grassmannians are Hermitian, the relation between theta lifts and Borcherds products (which lies in the heart of the definition of the latter functions in \cite{[B]}) reproduces the result of \cite{[M]}, independently of the Koecher principle.

In another direction, it is known that several lifts, like those of Shimura, Doi--Naganuma, Maa\ss, and others, can be viewed as theta lifts (still in the Hermitian setting but now with a non-trivial polynomial). Section 14 of \cite{[B]} extends these lifts to go from nearly holomorphic modular forms to produce meromorhpic automorphic forms with poles along special cycles (see \cite{[LZ]} for the exact description in the Shimura lift case). Some results about restrictions between such lifts are known, at least in the smooth case (see, e.g., \cite{[V1]}). Our main result generalizes these relations, and extends them to the meromorphic case. More explicitly, assume that $\Phi_{L}$ is the theta lift of the weakly holomorphic modular form $F$ of weight $1-\frac{n}{2}+m$, which is an automorphic form of weight $m$ on the $n$-dimensional Grassmannian of $L$, and take a primitive sub-lattice $M$ yielding a sub-Grassmannian of dimension $l$. Then by subtracting all the poles of $\Phi_{L}$ that contain the Grassmannian of $M$ we obtain a function whose restriction to the latter Grassmannian is the theta lift of the weakly holomorphic theta contraction of $F$, of weight $1-\frac{l}{2}+m$. Note that this result cannot be obtained by local considerations of $\Phi_{L}$, since in general its main terms there will be the poles that we subtract before restricting.

\smallskip

The paper is divided into two sections. Section \ref{DecTheta} proves the formulae involving theta functions (including the modularity of the modified ones), while Section \ref{ResLift} considers the pairings, proves the main result, and gives some elementary examples.

\section{Decomposing Generalized Theta Functions \label{DecTheta}}

Let $L$ be an even lattice of signature $(b_{+},b_{-})$. For $\lambda$ and $\mu$ in $L$ we write their pairing as $(\lambda,\mu)$, and after we shorten $(\lambda,\lambda)$ to $\lambda^{2}$, the quadratic form on $L$ takes $\lambda$ to $\frac{\lambda^{2}}{2}$ (with is in $\mathbb{Z}$ since $L$ is even). The the \emph{dual lattice} \[L^{*}:=\operatorname{Hom}(L,\mathbb{Z})=\{\nu \in L_{\mathbb{R}}|(\nu,L)\subseteq\mathbb{Z}\} \subseteq L_{\mathbb{Q}} \subseteq L_{\mathbb{R}}\] contains $L$, and the quotient $D_{L}:=L^{*}/L$, called the \emph{discriminant group} of $L$, is finite. The discriminant group $D_{L}$ carries non-degenerate $\mathbb{Q}/\mathbb{Z}$-valued bilinear and quadratic forms, for which we carry over the notation from $L$.

The group $\operatorname{O}(L)$ of automorphisms of $L$ is a discrete subgroup of the Lie group $\operatorname{O}(L_{\mathbb{R}})\cong\operatorname{O}(b_{+},b_{-})$, and we denote by $\operatorname{SO}^{+}(L)$ the intersection of $\operatorname{O}(L)$ with the connected component $\operatorname{SO}^{+}(L_{\mathbb{R}})$ of this Lie group. The group $\operatorname{SO}^{+}(L)$ preserves $L^{*}$ hence acts on $D_{L}$, and the \emph{discriminant kernel}, or \emph{stable orthogonal group} of $L$, is \[\Gamma_{L}:=\ker\big(\operatorname{SO}^{+}(L)\to\operatorname{O}(D_{L})\big)=\big\{\mathcal{A}\in\operatorname{SO}^{+}(L)\big|\mathcal{A}\lambda-\lambda \in L\ \forall\lambda \in L^{*}\big\}.\]

The symmetric space of the Lie group $\operatorname{O}(L_{\mathbb{R}})$, or of its connected component $\operatorname{SO}^{+}(L_{\mathbb{R}})$, it the \emph{Grassmannian} of $L_{\mathbb{R}}$, which is denoted by
\begin{equation}
\operatorname{Gr}(L_{\mathbb{R}}):=\big\{L_{\mathbb{R}}=v_{+} \oplus v_{-}\big|v_{+}\gg0,\ v_{-}\ll0,\ v_{+} \perp v_{-}\big\}. \label{Grassdef}
\end{equation}
It is clear from the definition in Equation \eqref{Grassdef} that $\dim v_{\pm}=b_{\pm}$, and that each one of these spaces determines the other as its orthogonal complement. This is related to $\operatorname{Gr}(L_{\mathbb{R}})$ being a (connected) real manifold of dimension $b_{+}b_{-}$. For $v\in\operatorname{Gr}(L_{\mathbb{R}})$ and $\lambda \in L_{\mathbb{R}}$ we write $\lambda_{v_{\pm}}$ for the orthogonal projection of $\lambda$ onto the corresponding space.

\smallskip

The group $\operatorname{SL}_{2}(\mathbb{R})$ acts on the upper half-plane $\mathcal{H}:=\{\tau=x+iy\in\mathbb{C}|y>0\}$ via fractional linear transformations: $A=\binom{a\ \ b}{c\ \ d}$ takes $\tau\in\mathcal{H}$ to $A\tau:=\frac{a\tau+b}{c\tau+d}$, with the factor of automorphy $j(A,\tau):=c\tau+d$ which satisfies \[j(AB,\tau)=j(A,B\tau)j(B,\tau).\] This group admits a unique non-trivial double cover \[\operatorname{Mp}_{2}(\mathbb{R}):=\big\{(A,\phi)\big|A\in\operatorname{SL}_{2}(\mathbb{R}),\ \phi:\mathcal{H}\to\mathbb{C}\text{ holomorphic},\ \phi^{2}(\tau)=j(A,\tau)\big\}\]. The product rule is defined by \[(A,\phi)\cdot(B,\psi):=\big(AB,(\phi \circ B)\cdot\psi\big).\]

The inverse image of the discrete subgroup $\operatorname{SL}_{2}(\mathbb{Z})$ in $\operatorname{Mp}_{2}(\mathbb{R})$ is denoted by $\operatorname{Mp}_{2}(\mathbb{Z})$. It is generated by \[T:=\bigg(\binom{1\ \ 1}{0\ \ 1},1\bigg)\qquad\text{and}\qquad S:=\bigg(\binom{0\ \ -1}{1\ \ \ 0\ },\sqrt{\tau}\in\mathcal{H}\bigg),\] with the relations $S^{2}=(ST)^{3}=Z:=(-I,i)$ and $Z^{4}$ is trivial. Given a lattice $L$ as above (or just its discriminant form $D_{L}$), there is a representation $\rho_{L}$ of $\operatorname{Mp}_{2}(\mathbb{Z})$ on the space $\mathbb{C}[D_{L}]$, called the \emph{Weil representation} $\rho_{L}$ that is associated with $L$. It is defined, in the natural basis $\{\mathfrak{e}_{\gamma}\}_{\gamma \in D_{L}}$ of $\mathbb{C}[D_{L}]$, on the generators $T$ and $S$ of $\operatorname{Mp}_{2}(\mathbb{Z})$, by the formulae
\begin{equation}
\rho_{L}(T)\mathfrak{e}_{\gamma}=\mathbf{e}\big(\tfrac{\gamma^{2}}{2}\big)\mathfrak{e}_{\gamma}\qquad\text{and}\qquad\rho_{L}(S)\mathfrak{e}_{\gamma}=\frac{\mathbf{e}\big(\frac{b_{-}-b_{+}}{8}\big)}{\sqrt{|D_{L}|}}\sum_{\delta \in D_{L}}\mathbf{e}\big(-(\gamma,\delta)\big)\mathfrak{e}_{\delta}, \label{Weildef}
\end{equation}
where we write $\mathbf{e}(w)$ for $e^{2\pi iw}$ for every $w$ in $\mathbb{C}$ or in $\mathbb{C}/\mathbb{Z}$. For the action of a general element of $\operatorname{Mp}_{2}(\mathbb{Z})$ under the representation $\rho_{L}$ from Equation \eqref{Weildef} see \cite{[Sch]}, \cite{[Str]}, or \cite{[Z]}.

\smallskip

Consider now an element $v\in\operatorname{Gr}(L_{\mathbb{R}})$. A polynomial $p_{v}$ on $\operatorname{Gr}(L_{\mathbb{R}})$ is called \emph{homogenous of degree $(m_{+},m_{-})$ with respect to $v$} if for every $\lambda \in L_{\mathbb{R}}$ and two constants $c_{\pm}$ we have the equality $p(c_{+}\lambda_{v_{+}}+c_{-}\lambda_{v_{-}})=c_{+}^{m_{+}}c_{-}^{m_{-}}p(\lambda)$ (i.e., in coordinates for $v_{\pm}$ combining to coordinates for $L_{\mathbb{R}}$ and the corresponding variables on $L_{\mathbb{R}}$, the polynomial $p$ is homogenous of degree $m_{\pm}$ in the variables associated with $v_{\pm}$). Let $\Delta_{v}$ be the Laplacian operator on $L_{\mathbb{R}}$ that is associated with the positive definite pairing (or \emph{majorant}) corresponding to $v$ (i.e., in variables as in the last remark it looks like $\sum_{i=1}^{b_{+}+b_{-}}\frac{\partial^{2}}{\partial x_{i}^{2}}$). Note that this is \emph{not} the natural Laplacian on $L_{\mathbb{R}}$ that is invariant under the action of $\operatorname{O}(L_{\mathbb{R}})$. For any $c\in\mathbb{R}$ we define the operator $e^{c\Delta_{v}}$ by the usual series $\sum_{j=0}^{\infty}\frac{c^{j}}{j!}\Delta_{v}^{j}$. Its operation on any polynomial $p$ on $\operatorname{Gr}(L_{\mathbb{R}})$ is well-defined (with no convergence issues), since $\Delta_{v}^{j}p=0$ for large enough $j$.

Take now $\tau=x+iy\in\mathcal{H}$, $v\in\operatorname{Gr}(L_{\mathbb{R}})$, and a polynomial $p_{v}$ on $L_{\mathbb{R}}$, which we assume to be homogenous of degree $(m_{+},m_{-})$ with respect to $v$. The \emph{vector-valued Siegel theta function} of $L$, with the polynomial $p_{v}$ at $v$, is
\begin{equation}
\Theta_{L}(\tau;v,p_{v}):=y^{\frac{b_{-}}{2}+m_{-}}\sum_{\lambda \in L^{*}}e^{-\frac{\Delta_{v}}{8\pi y}}(p)(\lambda)\mathbf{e}\bigg(\tau\frac{\lambda_{v_{+}}^{2}}{2}+\overline{\tau}\frac{\lambda_{v_{-}}^{2}}{2}\bigg)\mathfrak{e}_{\lambda+L}. \label{Thetadef}
\end{equation}
In fact, this function can be extended by adding two variables $\alpha$ and $\beta$ from $L_{\mathbb{R}}$, organized as a column vector $\binom{\alpha}{\beta} \in L_{\mathbb{R}}^{2}$, and consider the \emph{generalized Siegel theta function} $\Theta_{L}\big(\tau;\binom{\alpha}{\beta};v,p_{v}\big)$, which is defined to be
\begin{equation}
y^{\frac{b_{-}}{2}+m_{-}}\sum_{\lambda \in L^{*}}e^{-\frac{\Delta_{v}}{8\pi y}}(p_{v})(\lambda+\beta)\mathbf{e}\bigg(\tau\frac{(\lambda+\beta)_{v_{+}}^{2}}{2}+\overline{\tau}\frac{(\lambda+\beta)_{v_{-}}^{2}}{2}-\big(\lambda+\tfrac{\beta}{2},\alpha\big)\bigg)\mathfrak{e}_{\lambda+L}. \label{Thetagen}
\end{equation}
Recalling that $\operatorname{SL}_{2}(\mathbb{Z})\subseteq\operatorname{SL}_{2}(\mathbb{R})$, and with it $\operatorname{Mp}_{2}(\mathbb{Z})\subseteq\operatorname{Mp}_{2}(\mathbb{R})$, has a natural action on the column vectors $L_{\mathbb{R}}^{2}$ (essentially since the latter space is $\mathbb{R}^{2}\otimes_{\mathbb{R}}L_{\mathbb{R}}$, and the action is the usual one on the first component), Theorem 4.1 of \cite{[B]} establishes the following property of the theta function from Equation \eqref{Thetagen}.
\begin{thm}
For every $\tau\in\mathcal{H}$, $\binom{\alpha}{\beta} \in L_{\mathbb{R}}^{2}$, $v\in\operatorname{Gr}(L_{\mathbb{R}})$, polynomial $p_{v}$ that is homogenous of degree $(m_{+},m_{-})$ with respect to $v$, and $(A,\phi)\in\operatorname{Mp}_{2}(\mathbb{Z})$ we have the equality \[\Theta_{L}\big(A\tau;\textstyle{A\binom{\alpha}{\beta}};v,p_{v}\big)=\phi(\tau)^{b_{+}-b_{-}+2m_{+}-2m_{-}}\rho_{L}\big((A,\phi)\big)\Theta_{L}\big(\tau;\textstyle{\binom{\alpha}{\beta}};v,p_{v}\big).\] \label{modTheta}
\end{thm}
Since for $\alpha=\beta=0$ both theta functions reduce to the ones from Equation \eqref{Thetadef}, this case of Theorem \ref{modTheta} is the statement that for fixed such $v$ and $p_{v}$, the map taking $\tau\in\mathcal{H}$ to $\Theta_{L}(\tau;v,p_{v})$ is a (typically non-holomorphic) modular form of weight $\frac{b_{+}-b_{-}}{2}+m_{+}-m_{-}$ and representation $\rho_{L}$. When the polynomial $p$ is constant, the theta function (also the generalized one) are invariant under the action of $\Gamma_{L}$. However, for non-trivial polynomials, the behavior under the action of $\Gamma_{L}$ depends on the relation between the polynomials $p_{v}$ and $p_{\mathcal{A}v}\circ\mathcal{A}$ for $v\in\operatorname{Gr}(L_{\mathbb{R}})$ and $\mathcal{A}\in\Gamma_{L}$, which gives good modularity properties only in very special cases. Finally, the functions from Equations \eqref{Thetadef} and \ref{Thetagen} are holomorphic in $\tau$ precisely when $L$ is positive definite (so that $b_{-}=0$, $m_{-}=0$, and $\operatorname{Gr}(L_{\mathbb{R}})$ is a single point, with the associated Laplacian being the one of $L_{\mathbb{R}}$ itself), and the homogenous polynomial $p$ on $L_{\mathbb{R}}$ is harmonic (with respect to the canonical operator $\Delta_{v}$).

\medskip

In the classical case of $p=1$, the seesaw identity expresses $\Theta_{L}(\tau;v,1)$ using the theta functions associated with a primitive non-degenerate sub-lattice $M$ and its orthogonal complement, provided that $v$ is related to the presentation of $L_{\mathbb{R}}$ as $M_{\mathbb{R}} \oplus M_{\mathbb{R}}^{\perp}$ in a manner that we will soon define precisely. The basic idea is that if $L$ is the direct sum of $M$ and $M^{\perp} \cap L$ then $\Theta_{L}$ is the tensor product of the theta functions of the two sub-lattices, and otherwise it is obtained from this tensor product by a simple operation. We first generalize these formulae.

\smallskip

Consider thus a primitive non-degenerate sub-lattice $M$ of $L$, of some signature $(c_{+},c_{-})$ (where $0 \leq c_{\pm} \leq b_{\pm}$, of course), and denote the primitive sub-lattice $M^{\perp} \cap L$, of signature $(b_{+}-c_{+},b_{-}-c_{-})$, by $M^{\perp}_{L}$. We assume that the spaces $v_{\pm}$ associated with our element $v\in\operatorname{Gr}(L_{\mathbb{R}})$ intersect $M_{\mathbb{R}}$ in spaces $u_{\pm}$ of maximal dimension $c_{\pm}$, so that this restriction yields an element $u\in\operatorname{Gr}(M_{\mathbb{R}})$. This happens for the two spaces $v_{\pm}$ simultaneously, and is equivalent to the dimension of the intersection $u^{\perp}_{\pm}$ of $v_{\pm}$ with $M_{\mathbb{R}}^{\perp}$ having the maximal dimension $b_{\pm}-c_{\pm}$, determining an element $u^{\perp}\in\operatorname{Gr}(M_{\mathbb{R}})$. We denote this situation as $v=u \oplus u^{\perp}$, meaning that $v$ is the \emph{direct sum} of $u$ and $u^{\perp}$. Moreover, we assume that the polynomial $p_{v}$ is the product of polynomials $p_{u}$ and $p_{u^{\perp}}$, the former homogenous of degree $(n_{+},n_{-})$ with respect to $u$, so that the latter is homogenous of degree $(m_{+}-n_{+},m_{-}-n_{-})$ with respect to $u^{\perp}$. Then we have the following result, which was already used in many different references (e.g., \cite{[E]} and \cite{[M]}) in the case where $p=1$ and $\alpha=\beta=0$.
\begin{lem}
If $L$ is the direct sum $M \oplus M^{\perp}_{L}$ and $\alpha$ and $\beta$ are in $L_{\mathbb{R}}$ then we have, under the assumptions that $v=u \oplus u^{\perp}$ and $p_{v}=p_{u}p_{u^{\perp}}$, the equality \[\Theta_{L}\bigg(\tau;\binom{\alpha}{\beta};v,p_{v}\bigg)=\Theta_{M}\bigg(\tau;\binom{\alpha_{M_{\mathbb{R}}}}{\beta_{M_{\mathbb{R}}}};u,p_{u}\bigg)\otimes \Theta_{M^{\perp}_{L}}\bigg(\tau;\binom{\alpha_{M_{\mathbb{R}}^{\perp}}}{\beta_{M_{\mathbb{R}}^{\perp}}};u^{\perp},p_{u^{\perp}}\bigg)\] for every $\tau\in\mathcal{H}$, where $\alpha_{M_{\mathbb{R}}}$, $\alpha_{M_{\mathbb{R}}^{\perp}}$, $\beta_{M_{\mathbb{R}}}$, and $\beta_{M_{\mathbb{R}}^{\perp}}$ are the corresponding projections. \label{dirsum}
\end{lem}

\begin{proof}
When $L=M \oplus M^{\perp}_{L}$ we have $L^{*}=M^{*}\oplus(M^{\perp}_{L})^{*}$ and $D_{L}=D_{M} \oplus D_{M^{\perp}_{L}}$, and any element $\lambda \in L+\gamma \subseteq L^{*}$ is the sum of some $\mu \in M^{*}$ and $\nu \in (M^{\perp}_{L})^{*}$ (which are in the respective cosets in $D_{M}$ and $D_{M^{\perp}_{L}}$ that correspond to $\gamma$). Then our assumption on $v$ implies that \[(\lambda+\beta)_{v_{\pm}}=(\mu+\beta_{M_{\mathbb{R}}})_{u_{\pm}}+(\nu+\beta_{M_{\mathbb{R}}^{\perp}})_{u^{\perp}_{\pm}},\] and since the two terms are orthogonal, by taking norms we get \[(\lambda+\beta)_{v_{\pm}}^{2}=(\mu+\beta_{M_{\mathbb{R}}})_{u_{\pm}}^{2}+(\nu+\beta_{M_{\mathbb{R}}^{\perp}})_{u^{\perp}_{\pm}}^{2}.\] it is also clear that \[\big(\lambda+\tfrac{\beta}{2},\alpha\big)=\bigg(\mu+\frac{\beta_{M_{\mathbb{R}}}}{2},\alpha_{M_{\mathbb{R}}}\bigg)+\bigg(\nu+\frac{\beta_{M_{\mathbb{R}}^{\perp}}}{2},\alpha_{M_{\mathbb{R}}^{\perp}}\bigg).\] From $v=u \oplus u^{\perp}$ we deduce that the operator $\Delta_{v}$ is the sum of $\Delta_{u}$ acting on the variables from $M_{\mathbb{R}}$ and $\Delta_{u^{\perp}}$ acting on the variables from $M_{\mathbb{R}}^{\perp}$ (which commute), and thus our assumption on $p_{v}$ implies that \[e^{-\frac{\Delta_{v}}{8\pi y}}(p_{v})(\lambda+\beta)=e^{-\frac{\Delta_{u}}{8\pi y}}(p_{u})\big(\mu+\beta_{M_{\mathbb{R}}}\big) \cdot e^{-\frac{\Delta_{u}}{8\pi y}}(p_{u^{\perp}})\big(\nu+\beta_{M_{\mathbb{R}}^{\perp}}\big).\] In total, the summand from Equation \eqref{Thetagen} that corresponds to $\lambda$ is the product of those associated with $\mu$ and $\nu$. As the external power of $y$ is the correct one, the separation into the cosets in $D_{L}=D_{M} \oplus D_{M^{\perp}_{L}}$ yields the desired equality. This proves the lemma.
\end{proof}

\smallskip

Let $\Lambda$ now be an even lattice with discriminant form $D_{\Lambda}$. It is well-known, and very easy to check, that subgroups $H$ of $D_{\Lambda}$ that are \emph{isotropic} (i.e., on which the $\mathbb{Q}/\mathbb{Z}$ quadratic form, and thus also the bilinear form, vanishes), are in one-to-one correspondence with \emph{over-lattices} of $\Lambda$, namely subgroups $L\subseteq\Lambda_{\mathbb{R}}$ that contain $\Lambda$ and are still even lattices. When the over-lattice $L$ is associated with the subgroup $H$, the discriminant form $D_{L}$ is canonically isomorphic to $H^{\perp}/H$. This is so, because the dual lattice $L^{*}$ is the inverse image of $H^{\perp}$ in $\Lambda^{*}$. With this information, there are two operations between the spaces $\mathbb{C}[D_{\Lambda}]$ and $\mathbb{C}[D_{L}]$, defined by
\begin{equation}
\uparrow^{L}_{\Lambda}:\mathbb{C}[D_{L}]\to\mathbb{C}[D_{\Lambda}],\qquad\uparrow^{L}_{\Lambda}\mathfrak{e}_{\gamma}:=\sum_{\delta \in H^{\perp},\ \delta+H=\gamma}\mathfrak{e}_{\delta} \label{uparrowdef}
\end{equation}
in one direction, and in the other one we have
\begin{equation}
\downarrow^{L}_{\Lambda}:\mathbb{C}[D_{\Lambda}]\to\mathbb{C}[D_{L}],\qquad\downarrow^{L}_{\Lambda}\mathfrak{e}_{\delta}:=\begin{cases} \mathfrak{e}_{\delta+H} & \mathrm{in\ case\ }\delta \in H^{\perp} \\  0 & \mathrm{when\ }\delta \not\in H^{\perp}. \end{cases} \label{downarrowdef}
\end{equation}
The most important properties of these operators are the following ones.
\begin{lem}
The maps $\uparrow^{L}_{\Lambda}$ and $\downarrow^{L}_{\Lambda}$ from Equations \eqref{uparrowdef} and \eqref{downarrowdef} are maps of $\operatorname{Mp}_{2}(\mathbb{Z})$-representations, when the latter group acts on $\mathbb{C}[D_{L}]$ and $\mathbb{C}[D_{\Lambda}]$ via the respective Weil representations. The combination $\downarrow^{L}_{\Lambda}\circ\uparrow^{L}_{\Lambda}$ is $|H|$ times the identity map on $\mathbb{C}[D_{L}]$. \label{arrowprop}
\end{lem}
The first assertion of Lemma \ref{arrowprop} is proved in detail in Lemma 2.1 of \cite{[M]}, though it was known much earlier. The second assertion is also well-known, and follows from a simple computation. For the other composition, see Remark \ref{innprod} below.

A simple consequence, that generalizes Lemma 2.2 of \cite{[M]} (which was also well-known before) is as follows.
\begin{cor}
If $F:\mathcal{H} \times L_{\mathbb{R}}^{2}\to\mathbb{C}[D_{L}]$ and $G:\mathcal{H}\times\Lambda_{\mathbb{R}}^{2}\to\mathbb{C}[D_{\Lambda}]$ satisfy the functional equation from Theorem \ref{modTheta} with some weight $k$ and the representations $\rho_{L}$ and $\rho_{\Lambda}$ respectively, then the functions $\uparrow^{L}_{\Lambda}F$ and $\downarrow^{L}_{\Lambda}G$ also satisfy this equation, with the same weight $k$ and the respective representations $\rho_{\Lambda}$ and $\rho_{L}$. \label{arrowsmod}
\end{cor}
Corollary \ref{arrowsmod} follows directly from Lemma \ref{arrowprop} and the equality $L_{\mathbb{R}}=\Lambda_{\mathbb{R}}$, and its restriction to the zero vector of the square of this space is just the statement that $\uparrow^{L}_{\Lambda}$ and $\downarrow^{L}_{\Lambda}$ take modular forms to modular forms.

We can now express the theta function of an over-lattice $L$ of the lattice $\Lambda$ using that of $\Lambda$. Recalling that the Grassmannian from Equation \eqref{Grassdef} also depends only on the associated real vector space, all the parameters from Equation \eqref{Thetagen} can be taken to be the same for $L$ and $\Lambda$
\begin{lem}
If $L$ is an over-lattice of $\Lambda$ then for every element $v$ in the common Grassmannian, polynomial $p_{v}$, and two vectors $\alpha$ and $\beta$ in the common real vector space we have the equality $\Theta_{L}\big(\tau;\binom{\alpha}{\beta};v,p_{v}\big)=\downarrow^{L}_{\Lambda}\Theta_{\Lambda}\big(\tau;\binom{\alpha}{\beta};v,p_{v}\big)$. \label{downTheta}
\end{lem}

\begin{proof}
Since the parameters $v$, $p_{v}$, $\alpha$, and $\beta$ are the same, the scalar summand associated with $\lambda \in L^{*}\subseteq\Lambda^{*}$ is the same in $\Theta_{L}\big(\tau;\binom{\alpha}{\beta};v,p_{v}\big)$ and in $\downarrow^{L}_{\Lambda}\Theta_{\Lambda}\big(\tau;\binom{\alpha}{\beta};v,p_{v}\big)$. The restriction to elements of $L^{*}$ in the former theta function, and the gathering of the coefficients associated with different elements of the same $H$-coset together, correspond precisely to the definition in Equation \eqref{downarrowdef}. This proves the lemma.
\end{proof}
We can now obtain the formula relating the theta functions of $L$, $M$, and $M^{\perp}_{L}$ without the direct sum assumption on the lattices.
\begin{prop}
For $\tau$, $v$, $p_{v}$, $u$, $p_{u}$, $u^{\perp}$, $p_{u^{\perp}}$, $\alpha$, $\beta$, and their projections as in Lemma \ref{dirsum}, the theta function $\Theta_{L}\big(\tau;\binom{\alpha}{\beta};v,p_{v}\big)$ equals \[\Bigg\downarrow^{L}_{M \oplus M^{\perp}_{L}}\Bigg[\Theta_{M}\bigg(\tau;\binom{\alpha_{M_{\mathbb{R}}}}{\beta_{M_{\mathbb{R}}}};u,p_{u}\bigg)\otimes \Theta_{M^{\perp}_{L}}\bigg(\tau;\binom{\alpha_{M_{\mathbb{R}}^{\perp}}}{\beta_{M_{\mathbb{R}}^{\perp}}};u^{\perp},p_{u^{\perp}}\bigg)\Bigg].\] \label{sepTheta}
\end{prop}

\begin{proof}
The result follows from Lemmas \ref{dirsum} and \ref{downTheta}, when in the latter one we take $\Lambda=M \oplus M^{\perp}_{L}$. This proves the proposition.
\end{proof}

\medskip

Another formula, which turns out more useful, expresses the theta function $\Theta_{L}$ as the natural pairing of $\Theta_{M}$ with a more complicated theta function. Some incarnations of this function were known before: In the case where $b_{-}=1$, the sub-lattice $M$ is negative definite of rank 1, $v\in\operatorname{Gr}(L_{\mathbb{R}})$ is the element corresponding to the choice of $M$, $p=1$, and $\alpha=\beta=0$, this is the function denoted by $\Theta_{M,\lambda}$ in Theorem 10.6 of \cite{[B]} (in the opposite signature) or by $\Theta_{K,\omega}$ in Equation (4.27) of \cite{[BZ]}. This function is stated to be modular in these references (and others), a fact that is also used there. Related functions appear in some other signatures in, e.g., \cite{[V2]}, \cite{[M]}, and \cite{[SW]}, using representations related to the ones from \cite{[W1]}.

For introducing the pairing, we recall that given a lattice $L$, the lattice $L(-1)$ in which all the norms and pairings are inverted (hence also the signature) yields the discriminant form $D_{L}(-1)$ (with the same operation). We write the basis for $\mathbb{C}[D_{L(-1)}]$ as $\{\mathfrak{e}_{\gamma}^{*}\}_{\gamma \in D_{L}=D_{L(-1)}}$, meaning that we identify $\mathbb{C}[D_{L(-1)}]$ as the space dual to $\mathbb{C}[D_{L}]$ in such a way that the latter basis is the dual to the natural one for the latter space. This manifests itself in a pairing $\langle\cdot,\cdot\rangle_{L}$, with takes a vector $U\in\mathbb{C}[D_{L}]$ and $V\in\mathbb{C}[D_{L(-1)}]$ to the scalar $\langle U,V \rangle_{L}\in\mathbb{C}$. The Weil representation $\rho_{L(-1)}$ thus becomes identified with the representation $\rho_{L}^{*}$ dual to $\rho_{L}$. Moreover, the map $v_{+} \oplus v_{-} \mapsto v_{-} \oplus v_{+}$ is a canonical identification of $\operatorname{Gr}(L_{\mathbb{R}})$ with $\operatorname{Gr}\big(L_{\mathbb{R}}(-1)\big)$, which preserves the Laplacian operator $\Delta_{v}$ and inverts the order of the homogeneity degrees of polynomials. Since complex conjugation does not change these homogeneity degrees, it easily follows from the definition in Equation \eqref{Thetagen} that
\begin{equation}
\Theta_{L(-1)}\big(\tau;\textstyle{\binom{\alpha}{\beta}};v,p_{v}\big)=y^{\frac{b_{+}-b_{-}}{2}+m_{+}-m_{-}}\overline{\Theta_{L}\big(\tau;\textstyle{\binom{\alpha}{\beta}};v,\overline{p_{v}}\big)}. \label{negTheta}
\end{equation}

\begin{rmk}
It is fruitful, for many applications, to consider $\mathbb{C}[D_{L}]$ as an inner product space, by defining the natural basis to be orthonormal. Indeed, then $\rho_{L}$ becomes a unitary representation, and the operator $\uparrow^{L}_{\Lambda}$ from Equation \eqref{uparrowdef}, in case $L$ is an over-lattice of $\Lambda$, can be rescaled to an isometry. If $H=L/\Lambda \subseteq D_{\Lambda}$ is the corresponding isotropic subgroup, then the image of $\uparrow^{L}_{\Lambda}$ (or of this isometry) is the sub-representation consisting of the vectors that are supported on $H^{\perp}$ and in which the scalars are constant on cosets of $H$ in $H^{\perp}$. The composition $\downarrow^{L}_{\Lambda}\circ\uparrow^{L}_{\Lambda}$ is $|H|$ times the orthogonal projection onto this sub-representation of $\rho_{\Lambda}$. The pairing $\langle\cdot,\cdot\rangle_{L}$ is related to this inner product via the conjugate-linear map from $\mathbb{C}[D_{L}]$ to $\mathbb{C}[D_{L(-1)}]$ that takes $\mathfrak{e}_{\gamma}$ to $\mathfrak{e}_{\gamma}^{*}$ for every $\gamma \in D_{L}$. Many references define theta lifts using this inner product, but all the definitions are equivalent to the ones using our pairing via appropriate conventions and Equation \eqref{negTheta}. Since Proposition \ref{pairTheta} uses a pairing like $\langle\cdot,\cdot\rangle_{L}$, we choose to use this convention also for theta lifts. \label{innprod}
\end{rmk}

\smallskip

Let us now define the theta function in question. For a lattice $L$ and a primitive non-degenerate sub-lattice $M$, assume that $v=u \oplus u^{\perp}$ and $p_{v}=p_{u}p_{u^{\perp}}$ as in Lemma \ref{dirsum} and Proposition \ref{sepTheta}, and observe that the orthogonal projection from $L_{\mathbb{R}}$ onto $M_{\mathbb{R}}$ takes $L^{*}$ onto $M^{*}$ (this is the map dual to the injection of $M$ into $L$ by primitivity). Its composition with the natural map from $M^{*}$ onto $D_{M}$ clearly factors through a map from $L^{*}/M$ onto $D_{M}$, which we denote by $\pi_{M}$. Take $\tau\in\mathcal{H}$, $u^{\perp}\in\operatorname{Gr}(M_{\mathbb{R}}^{\perp})$, a polynomial $p_{u^{\perp}}$ that is homogenous of some degree that we write as $(m_{+}-n_{+},m_{-}-n_{-})$ with respect to $u^{\perp}$, and two vectors $\xi$ and $\eta$ of $M_{\mathbb{R}}^{\perp}$, and recall that the projection $\lambda_{M_{\mathbb{R}}^{\perp}}$ is well-defined for $\lambda \in L^{*}/M$. We then define the theta function \[\Theta_{L,M}\big(\tau;\textstyle{\binom{\xi}{\eta}};u^{\perp},p_{u^{\perp}}\big)\displaystyle:=y^{\frac{b_{-}-c_{-}}{2}+m_{-}-n_{-}}\sum_{\delta \in D_{M}}\sum_{\substack{\lambda \in L^{*}/M \\ \pi_{M}(\lambda)=\delta}}e^{-\frac{\Delta_{u^{\perp}}}{8\pi y}}(p_{u^{\perp}})\big(\lambda_{M_{\mathbb{R}}^{\perp}}+\eta\big)\times\]
\begin{equation}
\mathbf{e}\Bigg(\tau\frac{\big(\lambda_{M_{\mathbb{R}}^{\perp}}+\eta\big)_{u^{\perp}_{+}}^{2}}{2}+\overline{\tau}\frac{\big(\lambda_{M_{\mathbb{R}}^{\perp}}+\eta\big)_{u^{\perp}_{-}}^{2}}{2} -\big(\lambda_{M_{\mathbb{R}}^{\perp}}+\tfrac{\eta}{2},\xi\big)\Bigg)\mathfrak{e}_{\lambda+L}\otimes\mathfrak{e}_{\delta}^{*}, \label{compTheta}
\end{equation}
which we consider as $\mathbb{C}[D_{L}]\otimes_{\mathbb{C}}\mathbb{C}[D_{M(-1)}]$-valued. As in Equation \eqref{Thetadef}, we shall write just $\Theta_{L,M}\big(\tau;u^{\perp},p_{u^{\perp}}\big)$ for $\Theta_{L,M}\big(\tau;\binom{0}{0};u^{\perp},p_{u^{\perp}}\big)$. Extending the pairing $\langle\cdot,\cdot\rangle_{M}$ by allowing the second coordinate to be from $\mathbb{C}[D_{L}]\otimes_{\mathbb{C}}\mathbb{C}[D_{M(-1)}]$, so that the the pairing itself takes values in $\mathbb{C}[D_{L}]$, we have the following result.
\begin{prop}
If $\Theta_{L,M}\big(\tau;\binom{\xi}{\eta};u^{\perp},p_{u^{\perp}}\big)$ is defined as in Equation \eqref{compTheta} then we have the equality
\[\Theta_{L}\big(\tau;\textstyle{\binom{\alpha}{\beta}};v,p_{v}\big)\displaystyle=\Bigg\langle\Theta_{M}\bigg(\tau;\binom{\alpha_{M_{\mathbb{R}}}}{\beta_{M_{\mathbb{R}}}};u,p_{u}\bigg), \Theta_{L,M}\bigg(\tau;\binom{\alpha_{M_{\mathbb{R}}^{\perp}}}{\beta_{M_{\mathbb{R}}^{\perp}}};u^{\perp},p_{u^{\perp}}\bigg)\Bigg\rangle_{M}.\] \label{pairTheta}
\end{prop}

\begin{proof}
As in the proof of Lemma \ref{dirsum}, the scalar summand associated with $\lambda \in L^{*}$ in Equation \eqref{Thetagen} is the product of the summands associated with the projections $\lambda_{M_{\mathbb{R}}}$ and $\lambda_{M_{\mathbb{R}}^{\perp}}$, and we separate the sum according to the value $\mu \in M^{*}$ of $\lambda_{M_{\mathbb{R}}}$. We gather the values of $\mu$ together, and as the projection $\lambda_{M_{\mathbb{R}}}$ and the image in $D_{L}$ are invariant under adding an element of $M$, we can replace the sum over $\lambda \in L^{*}$ (or $\lambda$ in a coset $L+\gamma$) that project to $\mu \in M+\delta \subseteq M^{*}$ for $\delta \in D_{M}$ by a sum over $L^{*}/M$. All this expresses $\Theta_{L}\big(\tau;\binom{\alpha}{\beta};v,p_{v}\big)$ as $y^{\frac{b_{-}}{2}+m_{-}}$ times the sum over $\delta \in D_{M}$ of the coefficient multiplying $\mathfrak{e}_{\delta}$ in Equation \eqref{Thetagen} for $M$ (with the projections $\alpha_{M_{\mathbb{R}}}$ and $\beta_{M_{\mathbb{R}}}$) times the $\mathbb{C}[D_{L}]$-valued coefficient of $\mathfrak{e}_{\delta}^{*}$ in Equation \eqref{compTheta}. But this gives precisely the asserted pairing. This proves the proposition.
\end{proof}
As one example, Lemma 4.16 of \cite{[BZ]} is a special case of Proposition \ref{pairTheta}.

We now turn our attention to the behavior of the theta function from Equation \eqref{pairTheta}. We begin, once again, with the case of a direct sum of lattices.
\begin{lem}
Assume that $L=M \oplus M^{\perp}_{L}$, and take $u^{\perp}\in\operatorname{Gr}(M_{\mathbb{R}}^{\perp})$, a polynomial $p_{u^{\perp}}$, and vectors $\xi$ and $\eta$ in $M_{\mathbb{R}}^{\perp}$. Then we have the equality \[\Theta_{L,M}\big(\tau;\textstyle{\binom{\xi}{\eta}};u^{\perp},p_{u^{\perp}}\big)=\Theta_{M^{\perp}_{L}}\big(\tau;\textstyle{\binom{\xi}{\eta}};u^{\perp},p_{u^{\perp}}\big)\displaystyle\otimes\bigg(\sum_{\delta \in D_{M}}\mathfrak{e}_{\delta}\otimes\mathfrak{e}_{\delta}^{*}\bigg).\] \label{decomsum}
\end{lem}

\begin{proof}
As in the proof of Lemma \ref{dirsum}, we write $\lambda \in L^{*}$ the sum of $\mu \in M^{*}$ and $\nu\in(M^{\perp}_{L})^{*}$, and then the parameter $\lambda_{M_{\mathbb{R}}^{\perp}}$ appearing in Equation \eqref{pairTheta} is just $\nu$. Hence the scalar summand associated with $\lambda+M$ in that equation is just the one associated with $\nu$ in Equation \eqref{Thetagen} for $M^{\perp}_{L}$. Moreover, the projection $\pi_{M}(\lambda+M)$ is just the coset $\delta$ such that $\mu \in M+\delta$, so that the $\mathbb{C}[D_{L}]$-valued coefficient of $\mathfrak{e}_{\delta}$ is based only on those $\lambda$ in cosets $L+\gamma$ for which the $D_{M}$-part is $\delta$. As the scalar summand depends only on $\lambda_{M_{\mathbb{R}}^{\perp}}=\nu$, we indeed obtain an expression depending only on elements of $(M^{\perp}_{L})^{*}$ tensored with $\sum_{\delta \in D_{M}}\mathfrak{e}_{\delta}\otimes\mathfrak{e}_{\delta}^{*}$, and the previous expression was seen to be $\Theta_{M^{\perp}_{L}}\big(\tau;\binom{\xi}{\eta};u^{\perp},p_{u^{\perp}}\big)$. This proves the lemma.
\end{proof}

We shall also use the following analogue of Lemma \ref{downTheta}, whose proof is exactly the same as that of that lemma (but with the parameters being from $M_{\mathbb{R}}^{\perp}$ and its Grassmannian and the additional, constant vectors taken from $\mathbb{C}[D_{M(-1)}]$).
\begin{lem}
Assume that $L$ is an over-lattice of a lattice $\Lambda$, and that $M$ is a primitive non-degenerate sub-lattice of $\Lambda$ that remains primitive in $L$. Then $\Theta_{L,M}\big(\tau;\binom{\xi}{\eta};u^{\perp},p_{u^{\perp}}\big)=\downarrow^{L \oplus M(-1)}_{\Lambda \oplus M(-1)}\Theta_{\Lambda,M}\big(\tau;\binom{\xi}{\eta};u^{\perp},p_{u^{\perp}}\big)$. \label{Thetadown}
\end{lem}
We can now determine the desired behavior of the theta function from Equation \eqref{compTheta}.
\begin{thm}
For an even lattice $L$ of signature $(b_{+},b_{-})$, a primitive non-degenerate sub-lattice $M$ of $L$ having signature $(c_{+},c_{-})$, an element $u^{\perp}$ in the Grassmannian $\operatorname{Gr}(M_{\mathbb{R}}^{\perp})$ from Equation \eqref{Grassdef} for the lattice $M^{\perp}_{L}$, a polynomial $p_{u^{\perp}}$ on $M_{\mathbb{R}}^{\perp}$ that is homogenous of degree $(m_{+}-n_{+},m_{-}-n_{-})$ with respect to $u^{\perp}$, two vectors $\xi$ and $\eta$ in $M_{\mathbb{R}}^{\perp}$, and $\tau\in\mathcal{H}$, let $\Theta_{L,M}\big(\tau;\binom{\xi}{\eta};u^{\perp},p_{u^{\perp}}\big)$ be the function defined in Equation \eqref{compTheta}. Then if $(A,\phi)$ is in $\operatorname{Mp}_{2}(\mathbb{Z})$ then the value of $\Theta_{L,M}\big(A\tau;\textstyle{A\binom{\xi}{\eta}};u^{\perp},p_{u^{\perp}}\big)$ coincides with \[\phi(\tau)^{b_{+}-c_{+}-b_{-}+c_{-}+2m_{+}-2n_{+}-2m_{-}+2n_{-}} (\rho_{L}\otimes\rho_{M}^{*})\big((A,\phi)\big)\Theta_{L,M}\big(\tau;\textstyle{\binom{\xi}{\eta}};u^{\perp},p_{u^{\perp}}\big).\] \label{modcompTheta}
\end{thm}

\begin{proof}
Assume first that $L=M \oplus M^{\perp}_{L}$, where Lemma \ref{decomsum} expresses our theta function as a tensor product, and we claim that the constant vector from that lemma is invariant under the representation $\rho_{M}\otimes\rho_{M}^{*}$. It suffices to check on the generators $T$ and $S$ of the group $\operatorname{Mp}_{2}(\mathbb{Z})$ via the formulae from Equation \eqref{Weildef}. As for each $\delta \in D_{M}$, the vectors $\mathfrak{e}_{\delta}$ and $\mathfrak{e}_{\delta^{*}}$ are eigenvectors of $\rho_{M}(T)$ and $\rho_{M}^{*}(T)$ with multiplicative inverse eigenvalues, we deduce that $\mathfrak{e}_{\delta}\otimes\mathfrak{e}_{\delta}^{*}$ is invariant under $(\rho_{M}\otimes\rho_{M}^{*})(T)$, hence so is the sum. Now, since the signatures of $M$ and $M(-1)$ are opposite and $|D_{M}|=|D_{M(-1)}|$, we deduce that \[(\rho_{M}\otimes\rho_{M}^{*})(S)\bigg(\sum_{\delta \in D_{M}}\mathfrak{e}_{\delta}\otimes\mathfrak{e}_{\delta}^{*}\bigg)=\frac{1}{|D_{M}|}\sum_{\delta \in D_{M}}\sum_{\varepsilon \in D_{M}}\sum_{\epsilon \in D_{M}}\mathbf{e}\big((\delta,\epsilon-\varepsilon)\big)\cdot\mathfrak{e}_{\varepsilon}\otimes\mathfrak{e}_{\epsilon}^{*}.\] But after interchanging the order of summation, the inner sum over $\delta$ yields $|D_{M}|$ in case $\epsilon=\varepsilon$ and 0 otherwise, so that the right hand side reduces to the original vector $\sum_{\varepsilon \in D_{M}}\mathfrak{e}_{\varepsilon}\otimes\mathfrak{e}_{\varepsilon}^{*}$, as claimed. As Theorem \ref{modTheta} identifies $\Theta_{M^{\perp}_{L}}\big(A\tau;A\binom{\xi}{\eta};u^{\perp},p_{u^{\perp}}\big)$ with $\rho_{M^{\perp}_{L}}\big((A,\phi)\big)\Theta_{M^{\perp}_{L}}\big(\tau;\binom{\xi}{\eta};u^{\perp},p_{u^{\perp}}\big)$ times the required power of $\phi(\tau)$, we can tensor the first expression with $\sum_{\delta \in D_{M}}\mathfrak{e}_{\delta}\otimes\mathfrak{e}_{\delta}^{*}$ and the second one with its image under $(\rho_{M}\otimes\rho_{M}^{*})\big((A,\phi)\big)$ (which we have just seen to be the same), and obtain the desired equality because the asserted representation is $(\rho_{M^{\perp}_{L}}\otimes\rho_{M})\otimes\rho_{M}^{*}$. This yields the asserted equality in case $L=M \oplus M^{\perp}_{L}$.

In general, denote $M \oplus M^{\perp}_{L}$ by $\Lambda$, and then $L$ is an over-lattice of $\Lambda$ and we have proved the desired property for $\Theta_{\Lambda,M}$. But then we can express $\Theta_{L,M}$ using Lemma \ref{Thetadown}, and the assertion for it follows from that for $\Theta_{\Lambda,M}$ via Corollary \ref{arrowsmod}. This proves the theorem.
\end{proof}
In particular, Theorem \ref{modcompTheta} implies that for fixed $u^{\perp}$ and $p_{u^{\perp}}$, the expression $\Theta_{L,M}(\tau;u^{\perp},p_{u^{\perp}})$ is modular of weight $\frac{b_{+}-c_{+}-b_{-}+c_{-}}{2}+m_{+}-n_{+}-m_{-}+n_{-}$ and representation $\rho_{L}\otimes\rho_{M}^{*}$. Moreover, when $M^{\perp}$ is positive definite (so that $u^{\perp}$ is trivial, $c_{-}=b_{-}$, and $n_{-}=m_{-}$) and $p=p_{u^{\perp}}$ is harmonic, this is a \emph{holomorphic} modular form of weight $\frac{b_{+}-c_{+}}{2}+m_{+}-n_{+}$ and representation $\rho_{L}\otimes\rho_{M}^{*}$ (since so is $\Theta_{M^{\perp}_{L}}(\tau;u^{\perp},p_{u^{\perp}})$, and then use Lemmas \ref{decomsum} and \ref{Thetadown}). The case where $b_{+}=n-1$, $b_{-}=c_{-}=1$, $c_{+}=0$, and $p=1$ (so that $m_{+}=n_{+}=0$) produces Lemma 4.17 of \cite{[BZ]}, or equivalently the modularity assertion in Theorem 10.6 of \cite{[B]}. We remark that similar results hold for theta functions that are valued in differential forms, as in Sec 5 of \cite{[KM]} and other references.

\section{Restrictions of Theta Lifts \label{ResLift}}

The fact that the pairing $\langle\cdot,\cdot\rangle_{M}$ between $\mathbb{C}[D_{M}]$ and $\mathbb{C}[D_{M(-1)}]\cong\mathbb{C}[D_{M}]$ is a pairing of dual Weil representations immediately produces the following result.
\begin{lem}
Let $F$ and $G$ be (not necessarily holomorphic) modular forms of weights $k$ and $l$ and representations $\eta\otimes\rho_{M}$ and $\omega\otimes\rho_{M(-1)}$ respectively. Then the function $\tau\mapsto\big\langle F(\tau),G(\tau)\big\rangle_{M}$ transforms like a modular form of weight $k+l$ and representation $\eta\otimes\omega$. \label{pairMF}
\end{lem}
We shall also use the behavior of these pairings with respect to some simple operations.
\begin{lem}
Let $M$ and $L$ be two lattices, the latter being is an over-lattice of a lattice $\Lambda$. Take elements $U\in\mathbb{C}[D_{L(-1)}]$, $V\in\mathbb{C}[D_{M}]$, $W\in\mathbb{C}[D_{L} \oplus D_{M(-1)}]$, and $X\in\mathbb{C}[D_{\Lambda}]$. Then we have $\big\langle\downarrow^{L}_{\Lambda}X,U\big\rangle_{L}=\big\langle X,\uparrow^{L(-1)}_{\Lambda(-1)}U \big\rangle_{\Lambda}$, as well as \[\big\langle\langle V,W\rangle_{M},U\big\rangle_{L}=\langle W,U \otimes V\rangle_{L \oplus M(-1)}=\big\langle V,\langle W,U\rangle_{L}\big \rangle_{M}.\] Moreover, the arrow operators commute with the pairing associated with $M$, i.e., if in addition we have $Y\in\mathbb{C}[D_{\Lambda} \oplus D_{M(-1)}]$ then we get the equalities \[\big\langle V,\uparrow^{L \oplus M(-1)}_{\Lambda \oplus M(-1)}W\big\rangle_{M}=\uparrow^{L}_{\Lambda}\langle V,W\rangle_{M}\quad\mathrm{and}\quad\big\langle V,\downarrow^{L \oplus M(-1)}_{\Lambda \oplus M(-1)}Y\big\rangle_{M}=\downarrow^{L}_{\Lambda}\langle V,Y\rangle_{M}.\] \label{pairprop}
\end{lem}

\begin{proof}
Verifying all the statements reduces to the case where all vectors involved are standard basis vectors, where the equalities follow immediately from the definition of the pairing and of the arrow operators. This proves the lemma.
\end{proof}

\begin{rmk}
The constant vector from Lemma \ref{decomsum} corresponds, under the identification of $\mathbb{C}[D_{M}]\otimes\mathbb{C}[D_{M(-1)}]$ with $\operatorname{End}\big(\mathbb{C}[D_{M}]\big)$ via the duality, to $\operatorname{Id_{\mathbb{C}[D_{M}]}}$. This is equivalent to the fact that its pairing with every $V\in\mathbb{C}[D_{M}]$ yields back $V$ again, and explains why it indeed must be invariant under $\rho_{M}\otimes\rho_{M}^{*}$, as seen in the proof of Theorem \ref{modcompTheta}. This property can combine with Lemmas \ref{decomsum} and \ref{downTheta}, via the last equality from Lemma \ref{pairprop}, to deduce Proposition \ref{sepTheta} from Proposition \ref{pairTheta}. \label{IdinEnd}
\end{rmk}

\smallskip

We can use the equalities from Lemma \ref{pairprop} for expressing a pairing of a vector with a theta function, in the decomposable setting, as follows.
\begin{prop}
Let $M$ be a primitive non-degenerate sub-lattice of the lattice $L$, and assume that the element $v\in\operatorname{Gr}(L_{\mathbb{R}})$ is the direct sum of $u\in\operatorname{Gr}(M_{\mathbb{R}})$ and $u^{\perp}\in\operatorname{Gr}(M_{\mathbb{R}}^{\perp})$ as defined above. Take $\tau\in\mathcal{H}$, vectors $\alpha$ and $\beta$ in $L_{\mathbb{R}}$, and a polynomial $p_{v}$ that is the product of polynomials $p_{u}$ and $p_{u^{\perp}}$ as above. Then, given a vector $U\in\mathbb{C}[D_{L(-1)}]$, the pairing $\big\langle\Theta_{L}\big(\tau;\binom{\alpha}{\beta};v,p_{v}\big),U\big\rangle_{L}$ with the generalized theta function from Equation \eqref{Thetagen} equals both \[\Bigg\langle\Theta_{M}\bigg(\tau;\binom{\alpha_{M_{\mathbb{R}}}}{\beta_{M_{\mathbb{R}}}};u,p_{u}\bigg)\otimes \Theta_{M^{\perp}_{L}}\bigg(\tau;\binom{\alpha_{M_{\mathbb{R}}^{\perp}}}{\beta_{M_{\mathbb{R}}^{\perp}}};u^{\perp},p_{u^{\perp}}\bigg),\big\uparrow^{L}_{M \oplus M^{\perp}_{L}}U\Bigg\rangle_{L}\] and \[\Bigg\langle\Theta_{M}\bigg(\tau;\binom{\alpha_{M_{\mathbb{R}}}}{\beta_{M_{\mathbb{R}}}};u,p_{u}\bigg),\bigg\langle \Theta_{L,M}\bigg(\tau;\binom{\alpha_{M_{\mathbb{R}}^{\perp}}}{\beta_{M_{\mathbb{R}}^{\perp}}};u^{\perp},p_{u^{\perp}}\bigg),U\bigg\rangle_{L}\Bigg\rangle_{M}.\] \label{exppair}
\end{prop}

\begin{proof}
The first formula is a consequence of Proposition \ref{sepTheta} and the first equality in Lemma \ref{pairprop}. For the second formula we use the second equality of the latter lemma, combined with Proposition \ref{pairTheta}. This proves the proposition.
\end{proof}
While the first expression in Proposition \ref{exppair}, which involves simpler operations and resembles visually the formulae that will soon be the theta contractions from \cite{[M]}, \cite{[SW]}, and others, the second expression there will be the more useful one for our purposes. Indeed, as Remark \ref{IdinEnd} shows, the presentation from Proposition \ref{pairTheta} is more fundamental than that of Proposition \ref{sepTheta}.

\smallskip

We now need to introduce some notation. Assume that $M$ is a primitive non-degenerate sub-lattice of the lattice $L$, of signature $(b_{+},b_{-})$. Then the choice of an element $u^{\perp}\in\operatorname{Gr}(M_{\mathbb{R}}^{\perp})$ determines an embedding
\begin{equation}
\iota_{u^{\perp}}:\operatorname{Gr}(M_{\mathbb{R}})\hookrightarrow\operatorname{Gr}(L_{\mathbb{R}}),\qquad \iota_{u^{\perp}}(u)=u \oplus u^{\perp}. \label{embGrML}
\end{equation}
When $M^{\perp}$ is definite, so that $\operatorname{Gr}(M_{\mathbb{R}}^{\perp})$ is trivial, there is a unique choice of $u^{\perp}$, and we identify $\operatorname{Gr}(M_{\mathbb{R}})$ with its image in $\operatorname{Gr}(M_{\mathbb{R}})$ under $\iota_{u^{\perp}}$ for this $u^{\perp}$. In particular, if $M$ is the intersection of $L$ with the subspace perpendicular to a non-isotropic primitive element $\lambda \in L$, and thus $M^{\perp}_{L}=\mathbb{Z}\lambda$, the corresponding subvariety of $\operatorname{Gr}(L_{\mathbb{R}})$ is denoted by $\lambda^{\perp}$, and it is known as a \emph{special divisor}.

We now choose, for every $v\in\operatorname{Gr}(L_{\mathbb{R}})$, a polynomial $p_{v}$ on $L_{\mathbb{R}}$ that is homogenous of degree $(m_{+},m_{-})$ with respect to $v$, such that the map $v \mapsto p_{v}$ is smooth. Let $F:\mathcal{H}\to\mathbb{C}[D_{L(-1)}]$ be a smooth function that transforms like a modular form of weight $\frac{b_{-}-b_{+}}{2}+m_{-}-m_{+}$ and representation $\rho_{L}^{*}$. Lemma \ref{pairMF} and Theorem \ref{modTheta} (with $\alpha=\beta=0$) imply that the function $\tau\mapsto\big\langle\Theta_{L}(\tau;v,p_{v}),F(\tau)\big\rangle_{L}$ is invariant under the action of $\operatorname{Mp}_{2}(\mathbb{Z})$. Recalling that \[\mathcal{F}:=\big\{\tau\in\mathcal{H}\big||x|\leq\tfrac{1}{2},\ |\tau|\geq1\big\}\quad\mathrm{and}\quad\mathcal{F}_{Y}:=\{\tau\in\mathcal{H}|y \leq Y\}\] are the fundamental domain for the action of $\operatorname{Mp}_{2}(\mathbb{Z})$ on $\mathcal{H}$ and its truncation (for $Y>1$), and that $\operatorname{CT}_{s=0}$ stands for the constant term of the Laurent expansion of a holomorphic function of $s\in\mathbb{C}$ at $s=0$, we recall the following result, which is proved in detail in \cite{[B]} for the nearly holomorphic case.
\begin{thm}
For $F$ in a large analytic class of vector-valued modular forms of weight $\frac{b_{-}-b_{+}}{2}+m_{-}-m_{+}$ and representation $\rho_{L}^{*}$, the expression \[\Phi_{L}(F;v,p_{v}):=\operatorname{CT}_{s=0}\lim_{Y\to\infty}\int_{\mathcal{F}_{Y}}\big\langle\Theta_{L}(\tau;v,p_{v}),F(\tau)\big\rangle_{L}y^{-s}\frac{dxdy}{y^{2}}\] defines a function of $v\in\operatorname{Gr}(L_{\mathbb{R}})$, which is smooth outside the union of the images $\lambda^{\perp}$ for non-isotropic vectors $\lambda \in L^{*}$. \label{liftdef}
\end{thm}
The function $\Phi_{L}$ from Theorem \ref{liftdef}, which is defined wherever the limit in $Y$ produces a holomorphic function of $s$ in a right half-plane which admits a meromorphic continuation to $s\in\mathbb{C}$, is called the \emph{(regularized) theta lift} of $F$ with respect to the function $v \mapsto p_{v}$. The local smoothness of the theta lift depends on the smoothness of the latter function. In the case of nearly holomorphic $F$, presented in Section 6 of \cite{[B]}, the only singularities are along $\lambda^{\perp}$ for $\lambda^{2}>0$ (note that Theorem \ref{liftdef} is defined using the pairing $\langle\cdot,\cdot\rangle_{L}$ while \cite{[B]} works with the pairing from Remark \ref{innprod}, whence the different sign here), but other modular forms, like weak Maa\ss\ forms, may produce singularities along the other $\lambda^{\perp}$'s as well.

\smallskip

We can now state and prove our main result, about restricting theta lifts.
\begin{thm}
Let $M$ be a primitive non-degenerate sub-lattice of $L$, of signature $(c_{+},c_{-})$, and consider $\operatorname{Gr}(M_{\mathbb{R}})$ as embedded in $\operatorname{Gr}(L_{\mathbb{R}})$ via the map $\iota_{u^{\perp}}$ from Equation \eqref{embGrML}, associated with the element $u^{\perp}\in\operatorname{Gr}(M_{\mathbb{R}}^{\perp})$. Assume that for every $v$ in that image, the polynomial $p_{v}$ is the product $p_{u}p_{u^{\perp}}$, where the polynomial $p_{u^{\perp}}$, which is homogenous of degree $(m_{+}-n_{+},m_{-}-n_{-})$ with respect to $u^{\perp}$, is constant along the sub-Grassmannian $\iota_{u^{\perp}}\big(\operatorname{Gr}(M_{\mathbb{R}})\big)$. Then the function \[\Theta_{(L,M)}(F;u^{\perp},p_{u^{\perp}}):\tau\mapsto\big\langle\Theta_{L,M}(\tau;u^{\perp},p_{u^{\perp}}),F(\tau)\big\rangle_{L}\] is modular of weight $\frac{c_{-}-c_{+}}{2}+n_{-}-n_{+}$ and representation $\rho_{M}^{*}$, and for every $v=u \otimes u^{\perp} \in \iota_{u^{\perp}}\big(\operatorname{Gr}(M_{\mathbb{R}})\big)$ we have the equality \[\Phi_{L}(F;v,p_{v})=\Phi_{M}\big(\Theta_{(L,M)}(F;u^{\perp},p_{u^{\perp}});u,p_{u}\big).\] \label{reslift}
\end{thm}

\begin{proof}
The modularity of $\Theta_{(L,M)}(F;u^{\perp},p_{u^{\perp}})$ follows directly from Theorem \ref{modcompTheta} and Lemma \ref{pairMF}. For the second one, Proposition \ref{pairTheta} expresses the pairing inside the integrand in the definition of $\Phi_{L}(F;v,p_{v})$ as the left hand side of the second equality in Lemma \ref{pairprop}, with $V=\Theta_{M}(\tau;u,p_{u})$, $W=\Theta_{L,M}(\tau;u^{\perp},p_{u^{\perp}})$, and $U=F(\tau)$, for each $\tau\in\mathcal{H}$. But transferring this expression to the form of the right hand side of that equality, we obtain the pairing inside the integrand of $\Phi_{M}\big(\Theta_{(L,M)}(F;u^{\perp},p_{u^{\perp}});u,p_{u}\big)$. Since the remaining parts of the evaluation of these theta lifts are independent of the choice of lattice, they produce the same value. This proves the theorem.
\end{proof}
The function $\Theta_{(L,M)}(F;u^{\perp},p_{u^{\perp}})$ is call the \emph{theta contraction} of $F$ (associated with $u^{\perp}$ and $p_{u^{\perp}}$), and thus Theorem \ref{reslift} means that the restriction of the theta lift of $F$ to $\operatorname{Gr}(M_{\mathbb{R}})$ (embedded into $\operatorname{Gr}(M_{\mathbb{R}})$ via $\iota_{u^{\perp}}$) is the theta lift of the theta contraction of $F$. Note, however, that given a point $u$ in $\operatorname{Gr}(M_{\mathbb{R}})$ near which the theta lift $\Phi_{M}$ is smooth, the theta lift $\Phi_{L}$ need not be smooth near $v=\iota_{u^{\perp}}(u)=u \otimes u^{\perp}$. Indeed, this point $v$ can lie in some sub-Grassmannian of the form $\lambda^{\perp}$ for an anisotropic element $\lambda\in(M^{\perp}_{L})^{*}$ along which $\Phi_{L}$ might be singular, and which will then contain all of $\iota_{u^{\perp}}\big(\operatorname{Gr}(M_{\mathbb{R}})\big)$. Therefore if one wishes to obtain $\Phi_{M}$ as a restriction of a function $\widetilde{\Phi}_{L}$ to a sub-manifold which is not fully contained in the singularities of $\widetilde{\Phi}_{L}$, one may obtain $\widetilde{\Phi}_{L}$ by subtracting from $\Phi_{L}$ all the singularities arising from $\lambda^{\perp}$ for anisotropic $\lambda\in(M^{\perp}_{L})^{*}$. This is indeed possible in many interesting cases, as we see below.

\smallskip

We present a few basic applications, on complex Grassmannians. When $b_{-}=2$, the real $2b_{+}$-dimensional Grassmannian $\operatorname{Gr}(L_{\mathbb{R}})$ is, canonically up to orientations of complex conjugation, a complex manifold of dimension $b_{+}$ (see, e.g., Section 13 of \cite{[B]} or Section 3 of \cite{[BZ]} for the details). Taking the lattice $M$ to be with $c_{-}=2$ as well, so that $M^{\perp}_{L}$ is positive definite, the image of $\operatorname{Gr}(M_{\mathbb{R}})$ in $\operatorname{Gr}(L_{\mathbb{R}})$ is well-defined (i.e., there is no choice of $u^{\perp}$), and $n_{-}=m_{-}$, the sub-Grassmannian $\operatorname{Gr}(M_{\mathbb{R}})$ is a complex sub-manifold of $\operatorname{Gr}(L_{\mathbb{R}})$. Assuming that $m_{+}=n_{+}=0$ and $p_{u^{\perp}}=1$, the theta functions $\Theta_{M^{\perp}_{L}}$ and $\Theta_{L,M}$ are holomorphic and have integral Fourier coefficients, and we shorthand the notation for the theta contraction $\Theta_{(L,M)}(F;u^{\perp},p_{u^{\perp}})$ of a modular form $F$, whose weight is now just $1-\frac{b_{+}}{2}+m_{-}$, to simply $\Theta_{(L,M)}(F)$, of weight $1-\frac{c_{+}}{2}+m_{-}$.

In the case $p_{v}=p_{u}=1$, with $m_{-}=1$, and $F$ weakly holomorphic with integral Fourier coefficients, Theorem 13.3 of \cite{[B]} shows that $\Phi_{L}$ is, up to some globally smooth expressions of moderate growth, the logarithm of a meromorphic automorphic form $\Psi_{L}$ with respect to $\Gamma_{L}$ on $\operatorname{Gr}(L_{\mathbb{R}})$, called the \emph{Borcherds product} associated with $F$. Moreover, the (logarithmic) singularities of $\Phi_{L}$ determine the divisor of $\Psi_{L}$. Then $\Theta_{(L,M)}(F)$ is also weakly holomorphic with integral Fourier coefficients, and thus $\Phi_{M}$ is the logarithm of the Borcherds product $\Psi_{M}$ that is associated with $F$. Now, the singularities of $\Psi_{L}$ that cover $\operatorname{Gr}(M_{\mathbb{R}})$ completely are related to the zeros and poles of $\Phi_{L}$ that cover it, and subtracting the singularities of the former corresponds, under exponentiation, to the division by linear functions appearing in the definition of the quasi-pullback in \cite{[M]}. Therefore this case of Theorem \ref{reslift} reproduces the main theorem of \cite{[M]}, also for the case where the Koecher principle no longer holds.

Consider now the case where $p_{v}=p_{u}$ is the polynomial $(X+iY)^{m_{-}}$ for an oriented, orthogonal, equi-normed basis for $v_{-}$ (this is defined only on the tautological bundle of the complex manifold $\operatorname{Gr}(L_{\mathbb{R}})$, but this problem can be solved using sections or the definition of automorphic forms using this bundle). If $F$ is still weakly holomorphic (hence so is $\Theta_{(L,M)}(F)$ again), then Theorem 14.3 of \cite{[B]} states that $\Phi_{L}$ itself is a meromorphic automorphic form of weight $m$ with respect to $\Gamma_{L}$, and it has poles of order $m$ along the sub-Grassmannians $\lambda^{\perp}$ for some $\lambda \in L^{*}$ with positive norm, which are now complex divisors. The same assertion holds for $\Phi_{M}$ arising from $\Theta_{(L,M)}(F)$, and Theorem \ref{reslift} implies that if we subtract all the polar parts of $\Phi_{L}$ that cover $\operatorname{Gr}(M_{\mathbb{R}})$ completely, the restriction of the resulting function to $\operatorname{Gr}(M_{\mathbb{R}})$ is just $\Phi_{M}$. This generalizes some known results about true restrictions of holomorphic lifts of cusp forms, which appear (when $c_{-}=b_{-}-1$) \cite{[V1]} and \cite{[SW]}, and more generally \cite{[W2]}, to arbitrary dimensions and to meromorphic automorphic forms.

Recalling that when $c_{-}=1$ the Grassmannian $\operatorname{Gr}(M_{\mathbb{R}})$ becomes a real hyperbolic space of dimension $c_{+}$. It may be interesting to see what kind of functions do the restrictions of the meromorphic automorphic forms of weight $m$ from the previous paragraph, or the logarithm of Borcherds products from the preceding one, produce on such a hyperbolic space, and its variation with $u^{\perp}$ (which lies in another hyperbolic space, of dimension $b_{+}-c_{+}$).

\smallskip

Let us now make the operation of theta contraction from Theorem \ref{reslift} more explicit, using the natural bases for Weil representations. Consider a lattice $L$ and a primitive non-degenerate sub-lattice $M$, and viewing $L$ as an over-lattice of $\Lambda:=M \oplus M^{\perp}_{L}$, the quotient $H=L/\Lambda$ is an isotropic subgroup of $D_{\Lambda}=D_{M} \oplus D_{M^{\perp}_{L}}$. As stated in \cite{[M]} (and already proved in \cite{[N]}), the projections of $H$ onto both components $D_{M}$ and $D_{M^{\perp}_{L}}$ are injective (this is just the fact that $M$ and $M^{\perp}_{L}$ are primitive in $L$), and if we write their images as $H_{M}$ and $H_{M^{\perp}}$ respectively, then $H_{M^{\perp}} \cong H_{M}(-1)$ as (possibly degenerate) finite groups with $\mathbb{Q}/\mathbb{Z}$-valued bilinear and quadratic forms (this is equivalent to $H$ itself being isotropic).

We will examine the case where $H_{M}$ and $H_{M^{\perp}}$ are non-degenerate, to explain how this yields the formulae for theta contractions from \cite{[M]}, but generalized to our setting. This case includes the situations where $L=\Lambda$, or when one of the lattices is unimodular. The non-degeneracy yields decompositions \[D_{M}=H_{M} \oplus H_{M}^{\perp}\quad\mathrm{and}\quad D_{M^{\perp}_{L}}=H_{M^{\perp}} \oplus H_{M^{\perp}}^{\perp},\] and inside $D_{\Lambda}=D_{M} \oplus D_{M^{\perp}_{L}}$ we have
\begin{equation}
H^{\perp}=H_{M}^{\perp} \oplus H_{M^{\perp}}^{\perp} \oplus H\quad\mathrm{and\ thus}\quad D_{L} \cong H_{M}^{\perp} \oplus H_{M^{\perp}}^{\perp}. \label{Ddecom}
\end{equation}
In all the following formulae we suppress the variables $\tau$, $v$, $u$, $u^{\perp}$, $p_{v}$, $p_{u}$, $p_{u^{\perp}}$, $\alpha$, $\beta$, and their projections, but assume that the usual conditions $v=u \oplus u^{\perp}$ and $p_{v}=p_{u}p_{u^{\perp}}$ are satisfied. This allows us to write the theta functions $\Theta_{M}$ and $\Theta_{M^{\perp}_{L}}$ as
\begin{equation}
\sum_{\alpha \in H_{M}^{\perp}}\sum_{\gamma_{M} \in H_{M}}\theta_{M+\alpha+\gamma_{M}}\mathfrak{e}_{\alpha}\otimes\mathfrak{e}_{\gamma_{M}}\mathrm{\ and\ }\sum_{\beta \in H_{M^{\perp}}^{\perp}}\sum_{\gamma_{M^{\perp}} \in H_{M^{\perp}}}\theta_{M^{\perp}+\beta+\gamma_{M^{\perp}}}\mathfrak{e}_{\beta}\otimes\mathfrak{e}_{\gamma_{M^{\perp}}}, \label{MMperpcoor}
\end{equation}
and we combine the corresponding expression for $\Theta_{L}$ with Proposition \ref{sepTheta}, via the definition of $\downarrow^{L}_{\Lambda}$ and the expression for $H^{\perp}$ in Equation \eqref{Ddecom}. This gives
\begin{equation}
\theta_{L+\alpha+\beta}=\sum_{\gamma \in H}\theta_{M+\alpha+\gamma_{M}}\theta_{M^{\perp}+\beta+\gamma_{M^{\perp}}} \label{expthetaL}
\end{equation}
for every $\alpha \in H_{M}^{\perp}$ and $\beta \in H_{M^{\perp}}^{\perp}$, where for $\gamma \in H$ the notations $\gamma_{M}$ and $\gamma_{M^{\perp}}$ stand for the images of $\gamma$ in $H_{M}$ and $H_{M^{\perp}}$ respectively. The case of trivial $H$ in Equation \eqref{expthetaL} is the description in coordinates of Lemma \ref{dirsum}.

Now, Equation \eqref{MMperpcoor} and Lemma \ref{decomsum} imply that $\Theta_{\Lambda,M}$ is
\[\sum_{\beta \in H_{M^{\perp}}^{\perp}}\sum_{\gamma_{M^{\perp}} \in H_{M^{\perp}}}\sum_{\alpha \in H_{M}^{\perp}}\sum_{\gamma_{M} \in H_{M}}\theta_{M^{\perp}+\beta+\gamma_{M^{\perp}}}\mathfrak{e}_{\alpha}\otimes\mathfrak{e}_{\gamma_{M}}\otimes\mathfrak{e}_{\beta}\otimes\mathfrak{e}_{\gamma_{M^{\perp}}}\otimes \mathfrak{e}_{\alpha}^{*}\otimes\mathfrak{e}_{\gamma_{M}}^{*}\] (with $\gamma_{M}$ and $\gamma_{M^{\perp}}$ being independent indices), and Lemma \ref{Thetadown} then yields
\begin{equation}
\Theta_{L,M}=\sum_{\alpha \in H_{M}^{\perp}}\sum_{\beta \in H_{M^{\perp}}^{\perp}}\sum_{\gamma \in H}\theta_{M^{\perp}+\beta+\gamma_{M^{\perp}}}\mathfrak{e}_{\alpha}\otimes\mathfrak{e}_{\beta}\otimes\mathfrak{e}_{\alpha}^{*}\otimes\mathfrak{e}_{\gamma_{M}}^{*} \label{ThetaLMexp}
\end{equation}
(where $\gamma_{M}$ and $\gamma_{M^{\perp}}$ become the projections of $\gamma$ again). If the expansion of the modular form $F$, of weight $\frac{b_{-}-b_{+}}{2}+m_{-}-m_{+}$ and representation $\rho_{L}^{*}$, is as $\sum_{\alpha \in H_{M}^{\perp}}\sum_{\beta \in H_{M^{\perp}}^{\perp}}f_{L+\alpha+\beta}\mathfrak{e}_{\alpha}^{*}\otimes\mathfrak{e}_{\beta}^{*}$ (with $\tau$ suppressed again), then using Equation \eqref{ThetaLMexp} we get
\begin{equation}
\Theta_{(L,M)}(F)=\sum_{\alpha \in H_{M}^{\perp}}\sum_{\beta \in H_{M^{\perp}}^{\perp}}\sum_{\gamma \in H}f_{L+\alpha+\beta}\theta_{M^{\perp}+\beta+\gamma_{M^{\perp}}}\mathfrak{e}_{\alpha}^{*}\otimes\mathfrak{e}_{\gamma_{M}}^{*}. \label{contexp}
\end{equation}
Equation \eqref{contexp} generalizes Equation (3.7) from Example 3.11 of \cite{[M]}, and its special case of trivial $H$ yields, in particular, Equation (3.1) of that reference.

We consider a few simpler special cases. First, if $M$ is unimodular (so that in particular $H=\{0\}$), then Equation \eqref{ThetaLMexp} (or just Lemma \ref{decomsum}) identifies $\Theta_{L,M}$ with $\Theta_{M^{\perp}}$, and Equation \eqref{expthetaL} (or Lemma \ref{dirsum}) yields $\Theta_{L}=\theta_{M}\cdot\Theta_{M^{\perp}_{L}}$. More interesting is the case from Example 3.12 of \cite{[M]}, where $H$ surjects onto $D_{M^{\perp}_{L}}$, so that $H_{M^{\perp}}^{\perp}$ is trivial. Then we can identify the vector $\mathfrak{e}_{\gamma_{M}}^{*}$ for $\gamma_{M} \in H_{M}$ with $\mathfrak{e}_{\gamma_{M^{\perp}}}$ for $\gamma_{M^{\perp}} \in H_{M^{\perp}}=D_{M^{\perp}_{L}}$, and since Equation \eqref{Ddecom} identifies $D_{L}$ with $H_{M^{\perp}}$, Equation \eqref{ThetaLMexp} reduces to \[\Theta_{L,M}=\Theta_{M^{\perp}}\otimes\bigg(\sum_{\alpha \in D_{L}}\mathfrak{e}_{\alpha}\otimes\mathfrak{e}_{\alpha}^{*}\bigg).\] Equation \eqref{contexp} then reduces to the equality $\Theta_{(L,M)}(F)=F\otimes\Theta_{M^{\perp}}$ from Equation (3.8) of \cite{[M]}, which is in correspondence with the properties of the latter vector given in Remark \ref{IdinEnd}. The case where $H$ surjects also on $D_{M}$, i.e., $H_{M}^{\perp}$ is also trivial, is the case where $L$ is unimodular, considered in Example 3.13 of \cite{[M]}. Then $\Theta_{L,M}$ is again identified with $\Theta_{M^{\perp}_{L}}$, $F=f$ is scalar-valued, and $\Theta_{(L,M)}(F)=f\cdot\Theta_{M^{\perp}}$. Finally, if $M^{\perp}$ is unimodular (as in Example 3.14 of \cite{[M]}), then $\Theta_{L,M}$ is the scalar-valued function $\theta_{M^{\perp}_{L}}$ times $\sum_{\alpha \in D_{L}}\mathfrak{e}_{\alpha}\otimes\mathfrak{e}_{\alpha}^{*}$, and $\Theta_{(L,M)}(F)$ is indeed $\theta_{M^{\perp}_{L}} \cdot F$. Since $D_{L}=D_{M}$ is this case, Theorem \ref{reslift} determines the restriction of the lift of $F$ to the image of $\operatorname{Gr}(M_{\mathbb{R}})$ as the lift of $\theta_{M^{\perp}_{L}} \cdot F$.

\noindent\textsc{Einstein Institute of Mathematics, the Hebrew University of Jerusalem, Edmund Safra Campus, Jerusalem 91904, Israel}

\noindent E-mail address: zemels@math.huji.ac.il

\end{document}